\documentclass[12pt]{amsart} 
\usepackage{amscd}
\usepackage{amssymb}
\usepackage{a4wide}
\usepackage{amstext}
\usepackage{amsthm}
\usepackage{xcolor}
\usepackage{cite}
\usepackage[T1,T2A]{fontenc}
\usepackage[utf8]{inputenc}
\usepackage{url}
\usepackage{amsfonts}
\usepackage{amssymb, amsthm}
\usepackage{amsmath}
\usepackage{mathtools}
\usepackage{needspace}
\usepackage[pdftex]{graphicx}
\usepackage{hyperref}
\usepackage{datetime}
\usepackage{epigraph}
\usepackage{verbatim}
\usepackage{mathtools}
\usepackage{xcolor}
\numberwithin{equation}{section}

\newcommand{\T}{\mathbb{T}}
\newcommand{\N}{\mathbb{N}}

\newcommand{\D}{\mathbb{D}}

\newcommand{\summ}{\sum\limits}

\newcommand{\eps}{\varepsilon}

\renewcommand{\phi}{\varphi}

\newtheorem{Thm}{Theorem}[section]
\newtheorem{theorem}[Thm]{Theorem}

\newtheorem{lemma}[Thm]{Lemma}

\newtheorem{conjecture}[Thm]{Conjecture}

\begin{document}
\sloppy
\title[A contractive Hardy--Littlewood inequality]
{A contractive Hardy--Littlewood inequality}
\author{Aleksei Kulikov}
\address{Department of Mathematical Sciences, Norwegian University of Science and Technology, NO-7491 Trondheim, Norway
\newline {\tt lyosha.kulikov@mail.ru}
}

\begin{abstract} { We prove a contractive Hardy--Littlewood type inequality for functions from $H^p(\T)$, $0 < p \le 2$ which is sharp in the first two Taylor coefficients and asymptotically at infinity.
}
\end{abstract}

\maketitle

\section{Introduction}

The classical Hardy--Littlewood inequality \cite{MR1512359} says that for $f(z) = a_0 + a_1z + \ldots \in H^p(\T)$, $0 < p \le 2$, we have

\begin{equation}\label{hardy}
\summ_{n = 0}^\infty \frac{|a_n|^2}{(n+1)^{2/p-1}} \le C_p ||f||_p^2.
\end{equation}

In \cite{MR3858278} the following more precise version of this inequality was conjectured.

\begin{conjecture}
For the function $f(z) = a_0 + a_1z + a_2z^2 + \ldots \in H^p(\T), 0 < p \le 2$ we have
\begin{equation}\label{conj}
\summ_{n = 0}^\infty \frac{|a_n|^2}{c_{2/p}(n)} \le ||f||_p^2,
\end{equation}
where $c_\alpha(n) = \binom{n+\alpha - 1}{n}$.
\end{conjecture}

Despite vast numerical evidence this conjecture is currently proved only for $p = \frac{2}{k}, k\in \N$ by Burbea \cite{MR882113}, the case $p = 1$ being the famous Carleman inequality (see e.g.\cite{MR2263964} for a simple self-contained proof).

In  \cite{MR3304613} inequality \eqref{conj} was proved for the first two coefficients. Namely for the function $f\in H^p(\T), 0 < p \le 2$ we have $|f(0)|^2 + \frac{p}{2}|f'(0)|^2 \le ||f||_p^2$. In \cite{MR3870953}, by means of Wiessler's inequality\cite{MR578933}, the authors proved the following strengthening of this result.

\begin{theorem}
For the function $f(z) = a_0 + a_1z + a_2z^2 + \ldots \in H^p(\T), 0 < p \le 2$ we have
\begin{equation}
\summ_{n = 0}^\infty \frac{|a_n|^2}{\Phi_{2/p}(n)} \le ||f||_p^2,
\end{equation}
where $\Phi_\alpha(n) = c_{[\alpha]}(n) \left(\frac{\alpha}{[\alpha]}\right)^n$.
\end{theorem}

Note that $\Phi_\alpha(0) = c_\alpha(0) = 1, \Phi_\alpha(1) = c_\alpha(1) = \alpha$ but for $\alpha \notin \N$ these coefficients grow exponentially when $n$ goes to infinity.

In this paper we prove the following theorem which gives us an inequality that is also sharp in the first two terms but for $n\ge 2$ the weight decays as in the Hardy--Littlewood inequality \eqref{hardy}.

\begin{theorem}\label{my hardy}
For each $0 < p \le 2$ there exists $\eps_p > 0$ such that for all $f\in H^p(\T)$,  ${f(z) = a_0 + a_1z + a_2z^2 + \ldots}$ we have
\begin{equation}
|a_0|^2 + \frac{p}{2}|a_1|^2 + \eps_p \summ_{n = 2}^\infty \frac{|a_n|^2}{(n+1)^{2/p - 1}} \le ||f||_p^2.
\end{equation}
\end{theorem}

Note that the constant $\frac{p}{2}$ is optimal as can be seen from the function $f(z) = 1 + \eps z, \eps \to 0$.

The proof of this inequality is based on the following theorem which may be of independent interest.

\begin{theorem}\label{my truncation}
For $0 < p \le 2$ there exists $C_p' < \infty$ such that for all $f\in H^p(\T)$ we have
\begin{equation}
||f(z) - f(0) - f'(0)z||_p^2 \le C_p'(||f||_p^2 - |a_0|^2 - \frac{p}{2}|a_1|^2).
\end{equation}
\end{theorem}

Since this theorem is obviously true for $p =2$ we will prove it only for $0 < p < 2$. Moreover, the constants $C_p'$ will be uniformly bounded except possibly for $0 < p < \eps$  and $2 - \eps < p < 2$. It is easy to see that in the former case nonuniformity is unavoidable but we do not know what happens when $p$ is close to $2$.
\section{Weak form of Theorem \ref{my truncation}}

In this section we will prove the following lemma.
\begin{lemma}\label{weaker}
For every $0 < p \le 2$ there exists a constant $\gamma_p$ such that for all $f\in H^p(\T)$ we have
\begin{equation}
||f - f(0)||_p \le \gamma_p \sqrt{||f||_p^2 - |f(0)|^2}.
\end{equation}
\end{lemma}

In \cite[Lemma 2.2]{MR4024535} this is proved for $p \le 1$ and in \cite{MR3815251} this is proved for $1 < p \le 2$ (in \cite{MR3815251} this lemma is  proved even for $f\in L^p$, but with $\gamma_p\to \infty$ as $p\to 1$). Nevertheless we present here a simple uniform proof of this lemma.
\begin{proof}
Without loss of generality we may assume that $||f||_p = 1$. Let $n = \left [ \frac{2}{p} \right]$, $\frac{1}{q} + \frac{n}{2} = \frac{1}{p}$. We can decompose the function $f$ as a product  $f = f_0 f_1 \ldots f_n, f_0 \in H^q(\T), f_1, \ldots , f_n\in H^2(\T)$ such that $||f_0||_q =1,  {||f_k||_2 = 1}, k = 1, \ldots , n$.

Let $f_k(z) = a_k + g_k(z),\ g_k(0) = 0$. Note that $|a_k| \le 1$, $\prod\limits_{k = 0}^n |a_k| = |f(0)|$. Therefore $|a_k| \ge |f(0)|$. By orthogonality we have $||g_k||_2 \le \sqrt{1 - |f(0)|^2}$ and this inequality is valid even for $k = 0$ since $||f_0||_2 \le ||f_0||_q$. 

We have the following formula for $f - f(0)$:
\begin{equation}
f - f(0) = g_n \left(\prod_{k = 0}^{n - 1}f_k\right)  + g_{n-1}a_n \left(\prod_{k = 0}^{n - 2}f_k\right) + \ldots + g_1\left(\prod_{k = 2}^n a_k\right) f_0 + g_0 \left(\prod_{k = 1}^n a_k\right).
\end{equation}
For each of the first $n$ summands, by the obvious estimate $|a_k| \le 1$ and H{\"o}lder's inequality, we have $H^p$-norm is  bounded by $\sqrt{1 - |f(0)|^2}$. For the last summand we have $\prod_{k = 1}^n |a_k| \le 1$ and $||g_0||_p \le ||g_0||_2 \le \sqrt{1 - |f(0)|^2}$. Therefore by the triangle inequality (with the possible additional constant coming from the fact that $H^p(\T)$ for $p < 1$ is not a Banach space) we get $||f-f(0)||_p \le \gamma_p \sqrt{1 - |f(0)|^2}$.
\end{proof}

\section{Proof of Theorem \ref{my truncation} for functions without zeroes}
In this section we will prove the following theorem.

\begin{theorem}\label{no zeroes}
Let $0 < p < 2$ and $f\in H^p(\T)$ has no zeroes in $\D$. Then the conclusion of Theorem \ref{my truncation} holds for this function $f$.
\end{theorem}

For the proof of this theorem we will need the following result from \cite{MR4024535}.

\begin{theorem}
For $f\in H^p(\T)$ with $||f||_p = 1$ we have

\begin{equation}
|f'(0)| \le \kappa(p) = \begin{cases}1, &p \ge 1,\\
\sqrt{\frac{2}{p}}(1 - \frac{p}{2})^{1/p - 1/2},& 0 < p < 1.
\end{cases}
\end{equation}
\end{theorem}
Note that for all $0 < p < 2$ we have $\frac{p}{2}\kappa(p)^2 < 1$.

\begin{proof}[Proof of Theorem \ref{no zeroes}]
Without loss of generality we may assume that $||f||_p = 1$, $f(z) = a_0 + a_1z + \tilde f$. Note that $||\tilde f||_p \le A_p$ for some absolute constant $A_p < \infty$.

We fix $\delta_p > 0$ to be determined later and consider the following cases depending on the values of $|a_0|$ and $|a_1|$:

\begin{enumerate}
\item $|a_0| < \delta_p$,

\item $\delta_p \le |a_0| \le 1 - \delta_p, |a_1| < \delta_p$,

\item $\delta_p \le |a_0| \le 1 - \delta_p, |a_1| \ge \delta_p$,

\item $1 - \delta_p < |a_0|$.
\end{enumerate}

In the first three cases we will prove that $||f||_p^2 - |a_0|^2 - \frac{p}{2}|a_1|^2$ is greater than some absolute constant $\lambda_p > 0$ from which, by the inequality $||\tilde f||_p \le A_p$, the desired result follows.

In the first case we have $||f||_p^2 - |a_0|^2 - \frac{p}{2}|a_1|^2 \ge 1 - \delta_p^2 - \frac{p}{2}\kappa(p)^2$ which is positive if $\delta_p$ is small enough.

In the second case we have $||f||_p^2 -|a_0|^2 - \frac{p}{2}|a_1|^2 \ge 1 - (1-\delta_p)^2 - \delta_p^2 = 2(\delta_p - \delta_p^2) > 0$.

For the third case we will essentially repeat the proof of Lemma 1 from \cite{MR3304613}. We have $U(z) = f^{p/2}(z) = a_0^{p/2} + \frac{p}{2}a_0^{p/2-1}a_1z + \ldots$ with $||U||_2 = 1$. Therefore 
\begin{equation}
|a_0|^p + \frac{p^2}{4}|a_0|^{p-2}|a_1|^2\le 1.
\end{equation}

On the other hand we have 
\begin{equation}
\left(|a_0|^p + \frac{p^2}{4}|a_0|^{p-2}|a_1|^2\right)^{2/p} = |a_0|^2\left(1 + \left(\frac{p|a_1|}{2|a_0|}\right)^2\right)^{2/p} > |a_0|^2 \left(1 + \frac{p|a_1|^2}{2|a_0|^2}\right),
\end{equation}
where the last inequality is a Bernoulli's inequality $(1+t)^r > 1 + tr$ for $r > 1, t > 0$. Since we are on a compact set  $\delta_p \le |a_0| \le 1 - \delta_p$, $\delta_p \le |a_1| \le \kappa(p)$ and the functions are continuous we actually have a nonzero loss in the Bernoulli's inequality

\begin{equation}
|a_0|^2\left(1 + \left(\frac{p|a_1|}{2|a_0|}\right)^2\right)^{2/p} \ge |a_0|^2 \left(1 + \frac{p|a_1|^2}{2|a_0|^2}\right) + \lambda_p
\end{equation}
for some $\lambda_p > 0$. Therefore $1 \ge |a_0|^2 \left(1 + \frac{p|a_1|^2}{2|a_0|^2}\right) + \lambda_p = |a_0|^2 + \frac{p}{2}|a_1|^2 + \lambda_p$ as desired.

Now we turn to the fourth case which requires some additional ideas. Put $U(z) = f^{p/2}(z) = a_0^{p/2} + \frac{p}{2}a_0^{p/2 - 1}a_1 z + \tilde U(z) \in H^2(\T)$, $||U||_2 = 1$ (here we used that $f$ has no zeroes).

Denote $|a_0|^2 = 1 -\beta^2$, $||\tilde U||_2 = \eps$. Our goal now is to prove that $||\tilde f||_p \lesssim (\beta^2 + \eps)$.

Consider $V(z) = U(z)(1 - \frac{p}{2a_0}a_1z)$. We have
\begin{equation}
V(z) = a_0^{p/2} - \frac{p^2a_0^{p/2 - 2}}{4}a_1^2z^2 + \tilde U - \frac{p}{2a_0}a_1\tilde U z = a_0^{p/2} + \tilde V.
\end{equation}

Note also that by orthogonality it is easy to see from $||U||_2 = 1$ that $|a_1|, \eps \lesssim \beta$. Therefore we can bound $||\tilde V||_2 \lesssim \eps + \beta^2$. Thus, by Pythagoras's Theorem we have
\begin{equation}
||V||_2 = \sqrt{|a_0|^p + ||\tilde V||_2^2} \le \sqrt{|a_0|^p + O(\eps^2 + \beta^4)} = |a_0|^{p/2} + O(\eps^2 + \beta^4).
\end{equation}

We will now apply Lemma \ref{weaker} to the function $V^{2/p}$ ($V$ has no zeroes for small enough $\frac{|a_1|}{|a_0|}$, that is for small enough $\delta_p$):

\begin{equation}
||V^{2/p} - a_0||_p \lesssim \sqrt{||V||_2^{4/p} - |a_0|^2} \le \sqrt{|a_0|^2 + O(\eps^2 + \beta^4) - |a_0|^2} = O(\beta^2 + \eps).
\end{equation}

Now we are going to connect $V^{2/p} - a_0$ and $\tilde f$:
\begin{gather*}
V^{2/p} - a_0 = U^{2/p}(1 - \frac{p}{2a_0}a_1z)^{2/p} - a_0 = (a_0 + a_1z + \tilde f)(1 - \frac{a_1}{a_0}z + O(\beta^2)) - a_0 =\\ O(\beta^2) + \tilde f + \tilde f(a_1z + O(\beta^2)) = \tilde f + O(\beta^2) + O(\beta)\tilde f.
\end{gather*}

Therefore $||\tilde f|| = O(\beta^2 + \eps)(1 + O(\beta))^{-1} = O(\beta^2 + \eps)$, as required.

Since $||U||_2 = 1$ we have
\begin{equation}\label{3.5}
|a_0|^p + \frac{p^2}{4}|a_0|^{p-2}|a_1|^2 + \eps^2 = 1.
\end{equation}

Recall that in the end we want to prove that
\begin{equation}
|a_0|^2 + \frac{p}{2}|a_1|^2 + \eps_p||\tilde f||_p^2 \le 1.
\end{equation}

By our bound for $||\tilde f||_p$ it is enough to prove that
\begin{equation}
|a_0|^2 + \frac{p}{2}|a_1|^2 + c_p(\beta^4 + \eps^2) \le 1
\end{equation}
holds for some $c_p > 0$. Substituting the value of $|a_1|^2$ from \eqref{3.5} we get
\begin{equation}
|a_0|^2 + \frac{2}{p}|a_0|^{2-p}(1 - \eps^2 - |a_0|^p) + c_p(\beta^4 + \eps^2)\le 1.
\end{equation}

Choosing $c_p \le \frac{2}{p}(1 - \delta_p)^{2-p}$ we can neglect terms with $\eps$ and we are left with the inequality
\begin{equation}
(1 - \beta^2) + \frac{2}{p}(1 - \beta^2)^{1 - p/2}(1 - (1-\beta^2)^{p/2}) + c_p\beta^4 \le 1.
\end{equation}

Expanding the left-hand side via Taylor's formula we get
\begin{equation}
1 +\frac{p-2}{4}\beta^4 + c_p\beta^4 + O(\beta^6),
\end{equation}
and it is smaller than $1$ for $c_p < \frac{2-p}{4}$ and small enough $\beta$ (that is small enough $\delta_p$) since the constant in front of $\beta^4$ is negative.
\end{proof}
\section{Proof of Theorem \ref{my truncation}}
In this section we will finish the proof of Theorem \ref{my truncation} by taking into consideration the potential zeroes of the function $f$.

Let $f\in H^p(\T), ||f||_p = 1$. Write it as $f = Bg, ||g||_p = 1$, $g$ has no zeroes, $B = \prod_{n = 1}^N \frac{z - w_n}{1 - z\bar w_n}$ (obviously, it is enough to consider finite Blaschke products). Let $g(z) = a_0 + a_1z + \tilde g(z)$, $B(z) = b_0 + b_1z + \tilde B(z)$. We know that $|a_0|^2 + \frac{p}{2}|a_1|^2 + \eps_p ||\tilde g||_p^2 \le 1$ and we want to prove the same bound for $f$ (with possibly smaller $\eps_p$). 

Note that if $|f(0)| < \delta_p$ then as in the proof of Theorem \ref{no zeroes} we can prove the desired inequality. Therefore we can assume that $|f(0)| \ge \delta_p$. Since $|f(0)| \le |w_n|$ for all $n$ we have that $|w_n| \ge \delta_p$.

Put $f_k(z) = g(z) \prod_{n = 1}^k \frac{z-w_n}{1-z\bar w_n}$. Note that $|f_k(0)|\ge |f_N(0)| = |f(0)| \ge \delta_p$.

Carefully reading the proof of Lemma 1 from \cite{MR3304613} we can see that each factor $\frac{z - w_k}{1 - z\bar w_k}$ decreases $|f(0)|^2 + \frac{p}{2}|f'(0)|^2$ by at least $c_p(1 - |w_k|)$  for some $c_p > 0$, that is 
\begin{equation}\label{blashke reduction}
{|f_{k-1}(0)|^2 + \frac{p}{2}|f_{k-1}'(0)|^2} \ge |f_k(0)|^2 + \frac{p}{2}|f_{k}'(0)|^2 + c_p(1-|w_k|)
\end{equation}
(note that $c_p \to 0$ as $p\to 2$).

We have
\begin{equation}
|b_0| = \prod_{n = 1}^N |w_n| = \exp(\summ_{n = 1}^N \log |w_n|) \ge \exp(-C_p\summ_{n = 1}^N (1 - |w_n|)),
\end{equation}
where $C_p < \infty$ since all $w_n$ are bounded away from $0$. By orthogonality we have
\begin{equation}
||\tilde B||_p \le ||\tilde B||_2 \le \sqrt{1 - |b_0|^2} \le \sqrt{1 - \exp(-C_p\summ_{n = 1}^N(1 - |w_n|))} \le  \sqrt{C_p\summ_{n = 1}^N (1 - |w_n|)}.
\end{equation}

Let us now write $f(z) - f(0) - f'(0)z$ in terms of $B$ and $g$:
\begin{equation}
f(z) - f(0) - f'(0)z = b_1a_1z^2 + B(z)\tilde  g(z) + \tilde B(z) (a_0 + a_1z).
\end{equation}

Since Blaschke products are unimodular we have $||B\tilde g||_p = ||\tilde g||_p$. Since $|a_0|\le 1, |a_1| \le \kappa(p)$, the last term has $H^p$-norm at most $\alpha_p||\tilde B||_p$ for some $\alpha_p < \infty$. Finally, for $b_1$ we have again by orthogonality 
\begin{equation}
|b_1| \le \sqrt{1 - |b_0|^2} \le \sqrt{C_p\summ_{n = 1}^N (1 - |w_n|)}.
\end{equation}

Collecting everything we get
\begin{equation}\label{4.5}
||f(z) - f(0) - f'(0)z||_p \le A_p\left(||\tilde g||_p + \sqrt{\summ_{n = 1}^N (1 - |w_n|)}\right).
\end{equation}

On the other hand by \eqref{blashke reduction}
\begin{equation}\label{4.6}
|f(0)|^2 + \frac{p}{2}|f'(0)|^2 \le |a_0|^2 + \frac{p}{2}|a_1|^2 - c_p \summ_{n = 1}^N (1 - |w_n|)
\end{equation}
and by Theorem \ref{no zeroes}
\begin{equation}\label{4.7}
|a_0|^2 + \frac{p}{2}|a_1|^2 + \eps_p ||\tilde g||_p^2 \le 1.
\end{equation}

Now it is easy to see from \eqref{4.5}, \eqref{4.6}, \eqref{4.7} and the trivial inequality $(x+y)^2 \le 2x^2 + 2y^2$ that for some $\eps_p' > 0$ we have
\begin{equation}
|f(0)|^2 + \frac{p}{2}|a_1|^2 + \eps_p' ||f(z) - f(0) - f'(0)z||_p^2 \le 1,
\end{equation}
as required.

\section{Proof of Theorem \ref{my hardy}}

In this section we will deduce Theorem \ref{my hardy} from Theorem \ref{my truncation}.

We can rewrite inequality \eqref{hardy} as

\begin{equation}
\frac{1}{C_p}\summ_{n = 0}^\infty \frac{|a_n|^2}{(n+1)^{2/p - 1}} \le ||f||_p^2.
\end{equation}

Applying this to the function $\tilde f(z) = f(z) - f(0) - f'(0)z$ we get

\begin{equation}
\frac{1}{C_p}\summ_{n = 2}^\infty \frac{|a_n|^2}{(n+1)^{2/p - 1}}\le ||\tilde f||_p^2.
\end{equation}

Combining it with the bound from Theorem \ref{my truncation} we get

\begin{equation}
|a_0|^2 + \frac{p}{2}|a_1|^2 + \frac{1}{C_pC_p'}\summ_{n = 2}^\infty  \frac{|a_n|^2}{(n+1)^{2/p - 1}} \le ||f||_p^2.
\end{equation}

\subsection*{Acknowledgments} I would like to thank my advisor Kristian Seip for introducing me to this problem and 	for his constant support and encouragement, as well as for carefully reading this note and spotting some inaccuracies in the first version of it. I also would like to thank Dmitriy Stolyarov and Pavel Zatitskiy for fruitful discussions. This work was supported by Grant 275113 of the Research Council of Norway and by «Native towns», a social investment program of PJSC «Gazprom Neft».

\bibliographystyle{ieeetr}
\bibliography{Hardy-Littlewood_type_inequality}

\end{document}